\documentclass[a4paper,10pt]{amsart}
\usepackage[utf8]{inputenc}
\usepackage{lmodern}
\usepackage[english]{babel}
\usepackage[T1]{fontenc}
\usepackage{microtype} 
    \usepackage[OT2,OT1]{fontenc}
    \newcommand\cyr{%
    \renewcommand\rmdefault{wncyr}%
    \renewcommand\sfdefault{wncyss}%
    \renewcommand\encodingdefault{OT2}%
    \normalfont
    \selectfont}
    \DeclareTextFontCommand{\textcyr}{\cyr}
\usepackage{amsmath,amssymb,amsfonts,amsthm,amscd,latexsym,mathrsfs}
\usepackage[all]{xy}
\input xypic
\usepackage{color}
\usepackage{hyperref}
\usepackage{lipsum}
\usepackage{breqn}
\hypersetup{colorlinks=true,linkcolor=red,citecolor=blue}
\usepackage[hyperpageref]{backref} 
\usepackage{booktabs}
\usepackage{multirow}
\theoremstyle{plain}
\newtheorem{theorem}[subsection]{{\bf Theorem}}
\newtheorem*{theorem*}{{\bf Theorem}}
\newtheorem{corollary}[subsection]{{\bf Corollary}}
\newtheorem*{corollary*}{{\bf Corollary}}

\newtheorem{lemma}[subsection]{{\bf Lemma}}

\theoremstyle{definition}

\theoremstyle{remark}
\newtheorem{remark}[subsection]{{\it Remark}}

\numberwithin{equation}{section}

\newcommand{\gen}[1]{\langle #1 \rangle}

\begin{document}
\baselineskip=14pt
\title{On the dimension of the Schur multiplier of nilpotent Lie algberas}
\author[P. K. Rai]{Pradeep K. Rai}
\address[Pradeep K. Rai]{Department of Mathematics, Bar-Ilan University
Ramat Gan \\
Israel}
\email{raipradeepiitb@gmail.com}
\begin{abstract}
We give a bound on the dimension of the Schur multiplier of a finite dimensional nilpotent Lie algebra which sharpens the earlier known bounds. 
\end{abstract}
\maketitle
\section{Introduction}
\allowbreak{

Let $L$ be a finite dimensional Lie algebra. By $\gamma_i(L)$ and $Z_i(L)$ we denote the $i$-th term in the lower and upper central series of $L$ respectively. The dimension of the vector space $L$ is denoted by $\dim(L)$. For a subset $H$ of $L$, $\gen{H}$ denotes the vector subspace generated by $H$.

\vspace{.2cm}

The analogous theory for the Schur multiplier of a Lie algebra was developed in the dessertation of Batten and it has been further investigated in many recent papers \cite{Mehedi, Batten1, Batten2, Bosko1, Bosko2, Eshrati, Hardy, Niroomand_lie, Riyahi, Salemkar1, Salemkar2, Yankosky}. In most of the cases the theory runs parellel, though there are instances where the results for the Lie algebra does not match with the results for groups. For example it is a well known fact that there can be more than one cover for a finite group, but it has been proved in \cite{Batten1} that every finite dimensional Lie algebra has a unique cover. Also we have examples of groups generated by two elements having trivial Schur multiplier, but it has been proved by Bosko and Stitzinger in \cite{Bosko1} and also by Riyahi and Salemkar in \cite{Riyahi} that if a finite dimensional nilpotent Lie algebra has dimension greater than one, then its Schur multiplier is non-zero (however note a similar result proved by Golod and Safarevic in 1964, that if a finite $p$-group is generated by at least 4 elements, then its Schur multiplier is non-zero \cite{Golod}). 
Let $L$ be a finite dimensional $c$-step nilpotent Lie algebra with $\dim(L) =n$ and $\dim(\gamma_2(L)) = m$ and $M(L)$ denotes its Schur multiplier. Batten and others proved that $\dim (M(L)) \leq \frac{1}{2}n(n-1)$ \cite{Batten2}. Improving this bound Hardy and Stitzinger proved that $\dim (M(L)) \leq \frac{1}{2}n(n-1)- m$. This result was further improved by Yankosky \cite{Yankosky}. He established the following bound 

\begin{equation}
\dim (M(L)) \leq \dim(M(L/\gamma_c(L))) + \dim(L/\gamma_c(L))\dim(\gamma_c(L)) - \dim(\gamma_c(L)) \label{eq_yankosky}
\end{equation}

\noindent and derived that 

\[\dim(M(L)) \leq \frac{1}{2}(n+m)(n-m-1) = \frac{1}{2}n(n-1)-\frac{1}{2}m(m-1).\] 

\vspace{.2cm}
Improving this bound, Niroomand and Russo 
proved that if $m \geq 1$ then $\dim(M(L) \leq \frac{1}{2}(n-m-1)(n+m-2)+1$ \cite{Niroomand_lie}. We note here that this bound of Niroomand and Russo also follows from \cite[Theorem 3.10]{Eshrati} where Eshrati and others obtained a bound for the dimension of the Schur multiplier of a nilpotent n-Lie algebra.

In the following theorem we find an analogous bound to the bound in \cite[Theorem 1.1]{Rai} improving the bound of Niroomand and Russo.

\begin{theorem} \label{thm6}
Let $L$ be a $c$-step nilpotent Lie algebra with $\dim(L) = n$ and $\dim(\gamma_2(L)) = m \geq1$. Then $$\dim(M(L)) \leq  \frac{1}{2}(n-m-1)(n+m)- \sum_{i =2}^{\min(n-m,c)} n-m-i.$$
\end{theorem}

\section{Preriquisites}
 The following lemma gives the dimension of the Schur multiplier of an abelian Lie algebra.

\begin{lemma} \cite[Lemma 3]{Batten2} \label{lem_lie}
A Lie algebra L of dimension $n$ is abelian if and only if $\dim(M(L)) = \frac{1}{2}n(n-1)$.
\end{lemma} 

\begin{remark}\label{rmk1}
Let $F/R$ be a free presentation of a c-step nilpotent Lie algebra $L$. Eshrati et al. established the following equation (see the proof of \cite[Theorem 3.6]{Eshrati}):

\begin{equation}
\dim(\gamma_c(L)) + \dim(M(L)) = \dim(M(L/\gamma_c(L))) + \dim\Bigg(\frac{[F, \  \gamma_c(F)+R]}{[R, F]}\Bigg). \label{eq_Eshrati}
\end{equation}

They also defined the following surjective bilinear map 

\[f: \frac{F}{S_{c-1}} \times \frac{\gamma_c(F)+R}{R} \mapsto \frac{[F, \  \gamma_c(F)+R]}{[R, F]}\]

\[(x+ S_{c-1}, f_c+R) \mapsto [x, f_c][R, F],\]

\vspace{.2cm}
where $S_{c-1}$ is given by $Z_{c-1}(L) = S_{c-1}/R$. Clearly $\gamma_2(F) + R \leq S_{c-1}$. Hence this gives an epimorphism  $\lambda_c: \frac{L}{\gamma_2(L)} \otimes \gamma_c(L) \mapsto \frac{[F, \ \gamma_c(F)+R]}{[R, F]}$. 

Similarly the maps 
\[\frac{F}{\gamma_2(F)+R} \times \frac{\gamma_{i}(F)+R}{\gamma_{i+1}(F)+R} \mapsto \frac{[F, \  \gamma_{i}(F)+R]}{[\gamma_{i+1}(F)+R, F]}\]
can be defined, giving epimorphisms 
\[\lambda_{i}: \frac{L}{\gamma_2(L)} \otimes \frac{\gamma_{i}(L)}{\gamma_{i+1}(L)} \mapsto \frac{[F, \ \gamma_{i}(F)+R]}{[\gamma_{i+1}(F)+R, F]}\]
for $2 \leq i \leq c-1$.
\end{remark}




\section{Schur multiplier of nilpotent Lie algebras}
\begin{lemma}\label{lem_ellis}
Let $L$ be a Lie algebra. Then for $x_1, x_2, \ldots, x_{i+1} \in L$,
\begin{eqnarray*}
&& [[x_1, x_2, \cdots, x_i]_l, x_{i+1}] + [[x_{i+1}, [x_1, x_2, \cdots x_{i-1}]_l], x_i] \\
 && + [[[x_i, x_{i+1}]_r, [x_1, \cdots, x_{i-2}]_l], x_{i-1}] \\
&& + [[[x_{i-1}, x_i, x_{i+1}]_r, [x_1, x_2, \cdots, x_{i-3}]_l], x_{i-2}] 
+ \cdots + [[x_2, \cdots, x_{i+1}]_r, x_1] = 0 \\
\end{eqnarray*}
for $i \geq 3$, where 
\[[x_1, x_2, \cdots x_i]_r = [x_1, [\cdots [x_{i-2},[x_{i-1},x_i]]\ldots]\]
and 
\[[x_1, x_2, \cdots x_i]_l = [\ldots[[x_1, x_2], x_3], \cdots, x_i].\]
\end{lemma}

\begin{proof}
First we prove the lemma for $i =3$. In the following Jacoby idenity 
\[[\alpha_1, \alpha_2, \alpha_3]_l + [\alpha_2, \alpha_3, \alpha_1]_l + [\alpha_3, \alpha_1, \alpha_2]_l = 0,\]
putting $\alpha_1 = [x_1, x_2], \alpha_2 = x_3, \alpha_3 = x_4$, we get
\[[[x_1, x_2], x_3, x_4]_l + [x_3, x_4, [x_1, x_2]]_l + [x_4, [x_1, x_2], x_3]_l = 0.\]
Applying Jacoby identity again with $\alpha_1 = x_1, \alpha_2 = x_2$ and $\alpha_3 = [x_3, x_4]$, we have
\[[x_1, x_2, [x_3, x_4]]_l + [x_2, [x_3, x_4], x_1]_l + [[x_3, x_4], x_1, x_2]_l = 0.\]
This, together with the previous equation, gives that
\[[[x_1, x_2], x_3, x_4]_l + [x_2, [x_3, x_4], x_1]_l + [[x_3, x_4], x_1, x_2]_l + [x_4, [x_1, x_2], x_3]_l = 0.\]
This proves the lemma for $i = 3$.

Now suppose that the lemma is true for some $i = j \geq 3$, i.e., for $\alpha_1, \alpha_2, \ldots \alpha_{j+1} \in L$ we have 
\begin{eqnarray*}
&& [[\alpha_1, \alpha_2, \cdots, \alpha_j]_l, \alpha_{j+1}] + [[\alpha_{j+1}, [\alpha_1, \alpha_2, \cdots \alpha_{j-1}]_l], \alpha_j] \\
 && + [[[\alpha_j, \alpha_{j+1}]_r, [\alpha_1, \cdots, \alpha_{j-2}]_l], \alpha_{j-1}] \\
&& + [[[\alpha_{j-1}, \alpha_j, \alpha_{j+1}]_r, [\alpha_1, \alpha_2, \cdots, \alpha_{j-3}]_l], \alpha_{j-2}] 
+ \cdots + [[\alpha_2, \cdots, \alpha_{j+1}]_r, \alpha_1] = 0. \\
\end{eqnarray*}
Putting $\alpha_1 = [x_1, x_2], \alpha_2 = x_3, \alpha_3 = x_4, \ldots, \alpha_{j+1} = x_{j+2}$ in the above equation we get 
\begin{eqnarray*}
&& [[x_1, x_2, \cdots, x_{j+1}]_l, x_{j+2}] + [[x_{j+2}, [x_1, x_2, \cdots x_{j}]_l], x_{j+1}] \\
 && + [[[x_{j+1}, x_{j+2}]_r, [x_1, \cdots, x_{j-1}]_l], x_{j}] \\
&& + [[[x_{j}, x_{j+1}, x_{j+2}]_r, [x_1, x_2, \cdots, x_{j-2}]_l], x_{j-1}] 
+ \cdots + [[x_3, \cdots, x_{j+2}]_r, [x_1, x_2]] = 0.\\
\end{eqnarray*}
Applying Jacoby identity we have
\[[[x_3, \cdots, x_{j+2}]_r, [x_1, x_2]] = [[x_2, [x_3, \cdots, x_{j+2}]_r], x_1] + [[[x_3, \cdots, x_{j+2}]_r], x_1], x_2].\]
Using this in the prevoius equation we see that the lemma is true for $j+1$. This completes the proof.
\end{proof}

\begin{corollary}\label{cor1}
Let $\lambda_i$, for $ 2 \leq i \leq c$ be as defined in Remark \ref{rmk1}. Then 
\begin{eqnarray*}
\Psi_i(x_1, x_2, \ldots, x_{i+1}) := && \overline{[x_1, x_2, \cdots, x_i]_l} \otimes \overline{x_{i+1}} + \overline{[x_{i+1}, [x_1, x_2, \cdots x_{i-1}]_l]} \otimes \overline{x_i} \\
&& +\overline{[[x_i, x_{i+1}]_r, [x_1, \cdots, x_{i-2}]_l]} \otimes \overline{x_{i-1}} \\
&& + \overline{[[x_{i-1}, x_i, x_{i+1}]_r, [x_1, x_2, \cdots, x_{i-3}]_l]} \otimes \overline{x_{i-2}} \\
&& + \cdots + \overline{[x_2, \cdots, x_{i+1}]_r} \otimes \overline{x_1} \in \ker(\lambda_i).\\
\end{eqnarray*}
\end{corollary}

\vspace{.2cm}

{\bf Proof of Theorem \ref{thm6}}: Let $F/R$ be a free presentation of $L$ and $\lambda_i$, $2 \leq i \leq c$, be as defined in Remark \ref{rmk1}. Then by Equation \ref{eq_Eshrati} we have

\[\dim(M(L)) = \dim(M(L/\gamma_c(L))) - \dim(\gamma_c(L)) + \dim\Bigg(\frac{L}{\gamma_2(L)} \otimes \gamma_c(L) \Bigg) -\dim(\ker(\lambda_c)),\]

\noindent which further gives
\begin{equation}\label{eq2}
\dim(M(L)) = \dim(M(L/\gamma_c(L)))  + \Bigg(\dim\Bigg(\frac{L}{\gamma_2(L)}\Bigg) -1 \Bigg)\dim \gamma_c(L) -\dim(\ker(\lambda_c)).
\end{equation}

Notice for $2 \leq i \leq c$ that $\frac{L}{\gamma_i(L)}$ is an $(i-1)$-step nilpotent Lie algebra and $\frac{F}{\gamma_i(F) + R}$ is a free presentation of $\frac{L}{\gamma_i(L)}$.  Therefore a repeated application of Equation \ref{eq2} for the dimension of $M(L/\gamma_i(L))$, for $2 \leq i \leq c$ gives us that

\begin{equation}\label{eq3}
\dim(M(L) = \dim(M(L/\gamma_2(L))) + \Bigg(\dim\Bigg(\frac{L}{\gamma_2(L)}\Bigg) -1 \Bigg)\dim \gamma_2(L) - \sum_{i = 2}^{c} \dim(\ker(\lambda_i)).
\end{equation}

Let $U= \{x_1, x_2, \ldots, x_{n-m}\}$ be a minimal system of generators of $L$.
Fix $i \leq \min(n-m, c)$. Since $ i \leq c$, $\gamma_i(L)/\gamma_{i+1}(L)$ is non-trivial. We can choose a commutator $[y_1,y_2, \cdots, y_i]$ of weight $i$ such that $[y_1,y_2, \cdots, y_i] \notin \gamma_{i+1}(L)$ and $y_1, \ldots, y_i \in U$ . Since $i \leq n-m$, $U \backslash \{y_1,y_2, \cdots, y_i\}$ contains at least $n-m-i$ elements. Choose any $n-m-i$ elements $z_1, z_2, \ldots, z_{n-m-i}$ from $U \backslash \{y_1,y_2, \cdots, y_i\}$. Since $[y_1,y_2, \cdots, y_i] \notin \gamma_{i+1}(L)$ and $z_j \notin \{y_1,y_2, \cdots, y_i\}$, $\Psi_i(y_1, \ldots, y_i, z_j) \neq 1$. Notice that the set $\{\Psi_i(y_1, \ldots, y_i, z_j) \ \ | \ \ 1 \leq j \leq n-m-i\}$ is a linearly independent set 
because $\{x_1,x_2, \cdots, x_{n-m}\}$ is a  minimal system of generators for $L$. It follows from Corollary \ref{cor1} that $ \dim (\ker(\lambda_i)) \geq n-m-i$. Now putting this into Equation \ref{eq3} and applying Lemma \ref{lem_lie} we get the required result.
\qed

\begin{remark}
Note that $\lambda_i$s can be defined from $\frac{L}{\gamma_2(L)Z(L)} \otimes  
\frac{\gamma_i(L)}{\gamma_{i+1}(L)}$, and therefore bound in Theorem \ref{thm6} can be further improved by subtracting the number  $\dim(\frac{Z(L)}{\gamma_2(L) \cap Z(L)})\dim(\gamma_2(L))$.
\end{remark}

\vspace{.3cm}

{\bf Acknowledgements:} I am very much thankful to Prof. Graham Ellis for the discussion of the proof of Lemma \ref{lem_ellis}. I am also grateful to my post-doctoral supervisor Prof. Boris Kunyavski\u{\i} for his encouragement and support. This research was supported by Israel Council for Higher Education's fellowship program and by ISF grant 1623/16.


\begin{thebibliography}{999}

\bibitem{Mehedi}
M. Araskhan, \emph{The dimension of the c-nilpotent multiplier}, J. Algebra 386 (2013), 105-112.

\bibitem{Batten1}
P. Batten, K. Moneyhun and E. Stitzinger, \emph{On covers of Lie algebras}, Commun.
Algebra 24 (1996) 4301-4317.

\bibitem{Batten2}
P. Batten, K. Moneyhun and E. Stitzinger, \emph{On characterizing nilpotent Lie algebras
by their multipliers}, Commun. Algebra 24 (1996) 4319-4330.




\bibitem{Bosko1}
L.R. Bosko, E.L. Stitzinger, \emph{Schur multipliers of nilpotent Lie algebras}, Available at arXiv:1103.1812v1, (2010).

\bibitem{Bosko2}
L.R. Bosko, \emph{On Schur multipliers of Lie algebras and group of maximal class}, Int. J. Algebra Comput. 20 (2010), 807-821.



\bibitem{Eshrati}
M.Eshrati, F.Saeedi and H.Darabi, \emph{On the multiplier of nilpotent n-Lie algebras}, J. Algebra 450, (2016), 162-172.

\bibitem{Golod}
E.S. Golod, I.R. Safarevic, \emph{On class field towers}, Izv. Akad. Nauk. SSSR 28, 261-272.



\bibitem{Hardy}
P. Hardy and E. Stitzinger, \emph{On characterizing nilpotent Lie algebras by their multipliers t(L) = 3, 4, 5, 6}, Comm. Algebra 26 (1998), 3527-3539.









\bibitem{Niroomand_lie}
P. Niroomand F.G. Russo, \emph{A note on the Schur multiplier of nilpotent Lie algebras}, Comm. Algebra 39 (2011), 1293-1297.

\bibitem{Rai}
P.K. Rai, \emph{On the order of the Schur multiplier of $p$-groups}, Preprint.

\bibitem{Riyahi}
Z. Riyahi, A. R. Salemkar, \emph{A remark on the Schur multiplier of nilpotent Lie algebras}, J. Algebra 438 (2015), 1-6. 


\bibitem{Salemkar1}
A. R. Salemkar, V. Alamian and H. Mohammadzadeh, \emph{Some properties of the Schur
multiplier and covers of Lie algebras}, Comm. Algebra 36 (2008), 697-707.

\bibitem{Salemkar2}
A. R. Salemkar, B. Edalatzadeh and M. Araskhan, \emph{Some inequalities for the dimension
of the c-nilpotent multiplier of Lie algebras}, J. Algebra 322 (2009), 1575-1585.





\bibitem{Yankosky}
B. Yankosky, \emph{On the multiplier of a Lie algebra}, J. Lie Theory 13 (2003) 1-6.

\end{thebibliography}
\end{document}